\documentclass{amsart}

\usepackage{a4wide,amssymb,mathdots,mathrsfs,stmaryrd}

\usepackage[T1]{fontenc}

\usepackage[arrow,matrix,curve]{xy}

\let\le\leqslant
\let\ge\geqslant

\def\cub{{\scriptscriptstyle{\boxempty}}}

\def\id{\textrm{id}}

\def\xto#1{\xrightarrow{#1}}

\def\restr#1#2{\left.#1\right|_{#2}}

\newcommand{\nc}{\newcommand}

\nc{\sA}{{\mathscr A}}
\nc{\wt}{\widetilde} \nc{\bl}{\bullet} \nc{\al}{\alpha}
\nc{\sg}{\sigma} \nc{\vf}{\varphi} \nc{\om}{\omega}
\nc{\ve}{\varepsilon} \nc{\ol}{\overline} \nc{\lb}{\lambda}
\nc{\Lb}{\Lambda} \nc{\Gm}{\Gamma} \nc{\cP}{{\mathscr P}}
\nc{\sB}{{\mathscr B}}
\nc{\ul}{\underline} \nc{\os}{\overset} \nc{\us}{\underset}
\nc{\pa}{\partial} \nc{\wh}{\widehat} \nc{\sbs}{\subset}
\nc{\lra}{\longrightarrow} \nc{\all}{\allowdisplaybreaks}
\nc{\Ker}{\operatorname{Ker}} \nc{\Img}{\operatorname{Im}}
\nc{\Kan}{\operatorname{Kan}} \nc{\Hom}{\operatorname{Hom}}
\nc{\Imm}{\operatorname{Im}}

\newtheorem{theo}{Theorem}[section]
\newtheorem{prop}[theo]{Proposition}
\newtheorem{lem}[theo]{Lemma}
\newtheorem{coro}[theo]{Corollary}

\theoremstyle{definition}
\newtheorem{defi}[theo]{Definition}
\newtheorem{exmp}[theo]{Example}
\newtheorem{remk}[theo]{Remark}

\begin{document}


\title{Cubical Resolutions and Derived Functors}

\author{Irakli Patchkoria}

\begin{abstract}
We introduce  pseudocubical objects  with  pseudoconnections  in an arbitrary category, obtained from
the  Brown-Higgins  structure  of  a  cubical  object  with connections  by  suitably  relaxing  their
identities,  and construct a cubical analog of  the  Tierney-Vogel  theory of  simplicial  derived
functors.  The crucial point in the construction is that projective  precubical  resolutions  which
are naturally used to define our cubical derived  functors  possess  pseudodegeneracies and pseudoconnections.
The same fact  is  essentially used for proving that  in the case of an additive functor between abelian categories,
our theory  coincides with the classical relative theory  of derived  functors  by Eilenberg-Moore.
\end{abstract}



\maketitle

\section*{Introduction}

In [18] Tierney and Vogel for any functor $T:\sA\to\sB$, where $\sA$ is a category with finite limits and with a projective class $\cP$, and $\sB$ is an abelian category, have constructed derived functors and investigated relationships of their theory with other theories of derived functors. Namely, they have shown that if $\sA$ is abelian and $T$ is additive, then their theory coincides with the classical relative theory of Eilenberg-Moore [7], whereas if $\sA$ is abelian and $T$ is an arbitrary functor, then it gives a generalization of the theory of Dold-Puppe [6]. Besides, they showed that their derived functors are naturally isomorphic to the cotriple derived functors of Barr-Beck [3] if there is a cotriple in $\sA$ that realizes the given projective class $\cP$.

The key point in the construction of the derived functors by
Tierney and Vogel is that, using $\cP$-projective objects and
simplicial kernels, for every $A$ from $\sA$ a $\cP$-projective
resolution can be constructed, which is an $A$-augmented
pseudosimplicial object in $\sA$ and which for a given $A$ is
unique up to a presimplicial homotopy (according to the comparison
theorem for $\cP$-projective resolutions).

A natural question arises about constructing a cubical analog of
the theory of Tierney and Vogel. Exactly this is the purpose of
this paper.

Before turning to the content of the paper itself, let us say few
words about cubical objects and techniques. As it is well known,
simplicial methods are developed for long time and are
successfully used in algebra and topology. With less success, but
still also cubical techniques have been developed, which was
initiated on one hand by the systematic use of singular cubes in
the singular homology theory of topological spaces (see, for
example, [14]), and, on the other hand, by the papers of Kan [12,
13] which have related cubical sets to homotopy theory. Further
research (see e.~g. [1,2,4,5,8-11,16,17,19]) has also shown that
the cubical approach is interesting and important. Cubical objects
have a number of advantages compared to the simplicial ones. For
example, a (pre)cubical homotopy is given by a single morphism in
each dimension, whereas a (pre)simplicial homotopy requires
``many'' morphisms. On the other hand, significant disadvantages
of the cubical technique with respect to the simplicial one are
also apparent. For example, a cubical group, in fact even a
cubical abelian group, can fail to satisfy the Kan condition
[16,17]. In the context of elimination of these and other
disadvantages of the cubical theory, of extreme importance are
cubical objects with connections introduced by Brown and Higgins
[4]. These objects are cubical objects with extra degeneracies,
called connections. To stress importance and naturality of cubical
objects with connections it suffices to name e.~g. the following
three facts. The singular cubical complex of any topological space
has naturally defined connections. Next, Tonks in [19] has shown
that any cubical group with connections satisfies the Kan
condition. Finally, Brown and Higgins [5] have recently proved
that the category of cubical objects with connections in an
abelian category is equivalent to the category of nonnegative
chain complexes in the same category.

In this paper we introduce pseudocubical objects with
pseudoconnections, obtained from the Brown-Higgins structure of a
cubical object with connections by suitably relaxing their
identities. Projective precubical resolutions which we are using
to construct cubical derived functors possess pseudodegeneracies
and pseudoconnections and this fact is essentially used in the
construction of derived functors defined by us.

Now let us list the contents of the paper by sections.

In Section \ref{prelim} we recall the notions of presimplicial, pseudosimplicial and simplicial objects, precubical, cubical objects and cubical objects with connections, their morphisms, and the respective augmented versions of these notions. We also recall the definitions of normalization functors in the simplicial setting by Moore, and in the cubical setting by \'Swi\k atek [16]. Furthermore we recall the Kan cubical sets and their homotopy groups.

For any presimplicial object $S$ in an abelian category $\sA$ the
normalized chain complex of $S$, denoted by $I(S)$ in this paper,
is a chain subcomplex of the unnormalized chain complex of $S$,
denoted by $J(S)$. The well known Moore theorem says that if $S$
is a simplicial object in $\sA$, then this inclusion is a chain
homotopy equivalence (in fact this is valid for any
pseudosimplicial object in $\sA$). On the other hand, for any
precubical object $X$ in an abelian category $\sA$, one has the
canonical inclusion $i_X:M(X)\hookrightarrow N(X)$ of chain
complexes $M(X)$ and $N(X)$ in $\sA$ which are cubical analogs of
$I(S)$ and $J(S)$ (see [16, 17]) respectively, and which
functorially depend on $X$. In Section 2 we introduce
pseudocubical objects with pseudoconnections and prove that $i_X$
is a chain homotopy equivalence for any pseudocubical object $X$
with pseudoconnections (in particular for any cubical object $X$
with connections). Thus we obtain a cubical analog of the Moore
theorem. Then we establish some results which are not needed later
on; however they seem interesting by themselves. Let us mention
some of them. We indicate an alternative proof of the
aforementioned Brown-Higgins equivalence. Furthermore, we point
out that this equivalence is realized by the above functor $M$.
Next, we show that if $G$ is a Kan cubical group, then $\pi_n(G)$,
the $n$-th homotopy group of $G$, coincides with $H_n(M(G))$, the
$n$-th homology group of $M(G)$ (note that $M(G)$ and $H_n(M(G))$
are defined for any cubical group $G$ as well). Using this and the
cubical analog of the Moore theorem together with the
aforementioned result of Tonks, we get that $\pi_n(G)$ is
naturally isomorphic to $H_n(N(G))$ for any cubical abelian group
$G$ with connections.

Let $\sA$ be a category with finite limits and a projective class
$\cP$, $\sB$ an abelian category, and $T:\sA\to\sB$ an arbitrary
(covariant) finctor. In Section 3 we construct left cubical
derived functors $L_n^\cub T$, $\wt L_n^\cub T:\sA\to\sB$,
$n\ge0$, as follows. First we show how to build for any object $A$
from $\sA$ an $A$-augmented $\cP$-projective precubical
resolution, denoted $P\xto\pa A$, by means of cubical kernels and
$\cP$-projective objects. Then we prove the comparison theorem
which in particular says that $P\xto\pa A$ is unique up to
precubical homotopy for a given $A$. Furthermore, it is shown that
any $\cP$-projective precubical resolution $P\xto\pa A$ is in fact
an augmented pseudocubical object with pseudoconnections. We
define $L_n^\cub T(A)$, $n\ge0$, to be the $n$-th homology of
$N(T(P))$. Besides the comparison theorem, the fact that $P$ has
pseudodegeneracies is crucial in proving that $L_n^\cub T(A)$ are
well-defined and functorially depend on $A$. This contrasts with the construction of the derived functors by Tierney
and Vogel which does not use existence of pseudodegeneracies in
$\cP$-projective presimplicial resolutions (pseudodegeneracies of
$\cP$-projective presimplicial resolutions are essentially used
when the theory of Tierney and Vogel is compared with other
theories of derived functors). Further, we define $\wt L_n^\cub
T(A)$, $n\ge0$, to be the $n$-th homology group of $M(T(P))$. Now
this construction essentially uses the fact that $P$ is a
pseudocubical object with pseudoconnections, i.~e., this is
crucial for proving that $\wt L_n^\cub T(A)$ are well-defined and
functorial in $A$. The cubical analog of the Moore theorem proved
in the previous section shows that in fact there are isomorphisms
$L_n^\cub T(A)\cong\wt L_n^\cub T(A)$, $n\ge0$, which are natural
in $A$ and $T$.

Suppose $\sA$ is an abelian category with a projective class $\cP$, $\sB$ an abelian category, and $T:\sA\to\sB$ an additive (covariant) functor. Then one constructs, with respect to $\cP$, the left derived functors $L_nT:\sA\to\sB$ ($n\ge0$) of $T$ in the sense of Eilenberg-Moore [7]. On the other hand, since any abelian category admits finite limits, we can build $\cP$-projective precubical resolutions, and therefore can construct the cubical left derived functors $L_n^\cub T:\sA\to\sB$, $n\ge0$. In Section 4, using once again the cubical analog of the Moore theorem, we prove that if $\cP$ is closed [7] or, more generally, is closed with respect to retracts, then there are isomorphisms $L_n^\cub T(A)\cong L_n T(A)$, $A\in\sA$, $n\ge0$, which are natural in $A$ and $T$.

\section{Preliminaries}\label{prelim}

We begin with the following known definitions.

\begin{defi}
A presimplicial object $S$ in a category $\sA$ is a family of objects $(S_n\in\sA)_{n\ge0}$
together with face $\sA$-morphisms
\begin{align*}
\pa_i:&S_n\to S_{n-1}\ \ (n\ge1,\ 0\le i\le n)
\intertext{satisfying}%
\pa_i\pa_j&=\pa_{j-1}\pa_i\quad\quad\;\; \;i<j.
\intertext{A pseudosimplicial object is a presimplicial object %
together with pseudodegeneracy $\sA$-morphisms}%
s_i:&S_n\to S_{n+1}\ \ (0\le i\le n)
\intertext{satisfying}%
\pa_is_j&=\begin{cases}s_{j-1}\pa_i\quad&i<j,\\
\id\quad&i=j,\;j+1,\\
s_j\pa_{i-1}\quad&i>j+1.
\end{cases}
\intertext{A simplicial object is a pseudosimplicial object %
satisfying the identity}%
s_is_j&=s_{j+1}s_i\quad\quad\;\; \;i\le j.
\end{align*}
\end{defi}

A morphism $f:X\to X'$ between presimplicial objects in a category
$\sA$ is a family of  $\sA$-morphisms
$(f_n:X_n\to X'_n)_{n\ge0}$ which commute with the face operators.
If $X$ and $X'$ are (pseudo)simplicial objects, then the
$\sA$-morphisms $f_n$ must commute with the face and
(pseudo)de\-ge\-ne\-racy operators.

\begin{defi}
A precubical object $X$ in a category $\sA$ is a family of
objects $(X_n\in\sA)_{n\ge 0}$ together with $\sA$-morphisms
$$
\pa_i^0,\pa_i^1:X_n\to X_{n-1}\ \ (1\le i\le n)
$$
satisfying
\begin{align*}
\pa_i^\al\pa_j^\ve&=\begin{aligned}\pa_{j-1}^\ve\pa_i^\al\qquad&\quad\textrm{$i<j$,
$\al$, $\ve\in\{0,1\}$.}\end{aligned}
\end{align*}
The $\sA$-morphisms $\pa_i^0$ and $\pa_i^1$ are called face
operators.
\end{defi}

\begin{defi}[{[12]}]
A cubical object $X$ in a category $\sA$ is a family of
objects $(X_n\in\sA)_{n\ge0}$ together with $\sA$-morphisms
$$
\pa_i^0,\pa_i^1:X_n\to X_{n-1}
$$
as above and
$$
s_i:X_{n-1}\to X_n\ \ (1\le i\le n)
$$
which satisfy
\begin{align*}
\pa_i^\al\pa_j^\ve&=\left.\begin{aligned}\pa_{j-1}^\ve\pa_i^\al\qquad&\quad\textrm{$i<j$, $\al$, $\ve$ $\in\{0,1\}$,}\end{aligned}\right.\\
s_is_j&=\left.\begin{aligned}s_{j+1}s_i\qquad&\quad\textrm{\ $i\le
j$,}\end{aligned}\right.\\
\intertext{and}%
\pa_i^\al s_j&=\left\{\begin{aligned}
s_{j-1}\pa_i^\al\qquad&i<j,\\
\id\qquad\qquad&i=j,\\
s_j\pa_{i-1}^\al\qquad&i>j,
\end{aligned}\right.
\end{align*}
where $\al\in\{0,1\}$. The $\sA$-morphisms $s_i$  are
called degeneracy operators.
\end{defi}

\begin{defi} [{[4]}]
A cubical object $X$ in a category $\sA$ is said to have
connections if there are given  $\sA$-morphisms
$$
\Gm_i:X_n\to X_{n+1}\ \ (1\le i\le n)
$$
satisfying
\begin{align*}
\Gm_i\Gm_j&=\Gm_{j+1}\Gm_i\quad \quad\quad i\le j,\\
\Gm_i s_j&=\begin{cases}
s_{j+1}\Gm_i\quad\;\;&i<j,\\
s_i^2\quad&i=j,\\
s_j\Gm_{i-1}\quad&i>j,
\end{cases}\\
\pa_i^\al \Gm_j&=\begin{cases}
\Gm_{j-1}\pa_i^\al\quad&i<j,\;\;\al\in\{0,1\},\\
\id\quad&i=j,\;j+1,\;\;\al=0,\\
s_j\pa_j^\al\quad&i=j,\;j+1,\;\;\al=1,\\
\Gm_j\pa_{i-1}^\al\quad&i>j+1,\;\;\al\in\{0,1\}.
\end{cases}
\end{align*}
\end{defi}

\begin{exmp}[{[4]}]
The singular cubical complex $KX$ of a topological space $X$ is a cubical object with connections in the category of sets. The connections
$$
\Gamma_i:K_nX\to K_{n+1}X\ \ (n\ge1)
$$
are defined by
$$
\Gamma_i\left(f:[0,1]^n\to
X\right)(t_1,t_2,...,t_{n+1})=f(t_1,...,t_{i-1},\max(t_i,t_{i+1}),t_{i+2},...,t_{n+1}).
$$
\end{exmp}

A morphism $f:X\to X'$ between precubical objects in a category
$\sA$ is a family of  $\sA$-morphisms
$(f_n:X_n\to X'_n)_{n\ge0}$ which commute with the face operators.
If $X$ and $X'$ are cubical objects, then the
$\sA$-morphisms $f_n$ must commute with the face and
degeneracy operators; and if $X$ and $X'$
are cubical objects with connections, then the $\sA$-morphisms $f_n$
must commute with the faces, degeneracies and connections.

For any category $\sA$, let us denote by $pres(\sA)$ the category of presimplicial objects in
$\sA$, by $prec(\sA)$ the
category of precubical objects in $\sA$, by $c(\sA)$ the category of cubical objects in $\mathscr A$, and by $cc(\sA)$ the category of cubical objects with connections in $\sA$.

Let $\sA$ be an abelian category, and $Ch_{\ge0}(\sA)$ the category of
non-negatively graded chain complexes in $\sA$. We
essentially use the normalization functor
$$
N:prec(\sA)\to Ch_{\ge 0}(\sA)
$$
of \'Swi\k atek [16, 17] which is constructed as follows. If $X,Y\in
prec(\sA)$ and $f=(f_n:X_n\to Y_n)_{n\ge0}$ is a
precubical morphism, then define
\begin{align*}
&N_0(X)=X_0,\quad N_n(X)=\bigcap_{i=1}^n\Ker(\pa_i^1:X_n\to X_{n-1}),\quad n>0,\\
&\pa=\sum_{i=1}^n(-1)^{i+1}\pa_i^0:N_n(X)\to N_{n-1}(X),\quad n>0,\\
&N_n(f:X\to Y)=\restr{f_n}{N_n(X)},\ n\ge0.
\end{align*}

Of less importance for us is the functor
$$
C:prec(\sA)\lra Ch_{\ge0}(\sA)
$$
which is defined for arbitrary abelian category $\sA$ by
\begin{align*}
&C(X)_n=X_n,\ n\ge0,\\
&\pa=\sum_{i=1}^n(-1)^i(\pa_i^1-\pa_i^0):C(X)_n\lra C(X)_{n-1},\ n>0,\\
&C(f:X\to X')_n=f_n,\ n\ge0
\end{align*}
(see [16, 17]).

Let  $\sA$ be again an abelian category. Recall the definition of the
Moore normalization functor
$$
I:pres(\sA)\to Ch_{\ge 0}(\sA).
$$
Assume $S$, $S'\in pres (\sA)$ and $g=(g_n:S_n\to
S'_n)_{n\ge0}$ is a presimplicial morphism. Then
\begin{align*}
&I_0(S)=S_0,\quad I_n(S)=\bigcap_{i=0}^{n-1}\Ker(\pa_i:S_n\to S_{n-1}),\quad n>0,\\
&\pa=(-1)^n\pa_n:I_n(S)\to I_{n-1}(S),\quad n>0,\\
&I_n(g:S\to S')=\restr{g_n}{I_n(S)},\ n\ge0.
\end{align*}
One also has the functor
$$
J:pres(\sA)\to Ch_{\ge 0}(\sA)
$$
assigning to $S\in pres(\sA)$ its unnormalized chain
complex. More precisely,
\begin{align*}
&J_n(S)=S_n,\quad n\ge 0,\\
&\pa=\sum_{i=0}^n(-1)^i\pa_i:J_n(S)\to J_{n-1}(S),\quad n>0,\\
&J_n(g:S\to S')=g_n,\ n\ge0.
\end{align*}

An augmented precubical (resp.~presimplicial)  object in a
category $\sA$ is a precubical (resp.~pre\-sim\-pli\-cial)
object $X$ in $\sA$ together with an object $A\in\sA$ and an
$\sA$-morphism $\pa:X_0\to A$ satisfying $\pa\pa_1^0=\pa\pa_1^1$
(resp.~$\pa\pa_0=\pa\pa_1$). Such an object is denoted by
$X\xto\pa A$. A morphism between $X\xto\pa A$ and $X'\xto{\pa'}A'$
is a morphism
$$
\ol{f}=(f_n:X_n\to X'_n)_{n\ge0}
$$
between $X$ and $X'$ together with an $\sA$-morphism $f:A\to A'$
satisfying $f\pa=\pa' f_0$. Denote the cate\-\-gory of augmented
precubical (resp.~presimplicial) objects in $\sA$ by $aprec(\sA)$
(resp. $apres(\sA)$), by $aps(\sA)$ the category of augmented
pseudosimplicial objects in $\sA$, and by $as(\sA)$ the category
of augmented simplicial objects in $\sA$.

Assume $\sA$ is abelian. Then a nonnegative chain complex in
$\sA$  augmented over $A\in\sA$ is a nonnegative
chain complex $X$ in $\sA$ together with an $\sA$-morphism $\pa:X_0\to A$ satisfying $\pa\pa_1=0$. Denote the
category of augmented nonnegative chain complexes in $\sA$
by $aCh_{\ge 0}(\sA)$.

For any abelian category $\sA$, the functor $N:prec(\sA)\to
Ch_{\ge0}(\sA)$, respectively the functors $I$, $J:pres(\sA)\to
Ch_{\ge0}(\sA)$ extend in an obvious way to
$$
aN:aprec(\sA)\to aCh_{\ge 0}(\sA),
$$
respectively
$$
aI,aJ:apres(\sA)\to aCh_{\ge 0}(\sA).
$$

Let $\sA$ be an abelian category. Using the shifting
functor
$$
sh:aCh_{\ge 0}(\sA)\to Ch_{\ge 0}(\sA)
$$
(assigning to $X=(...\to X_n\xto{\pa_n}X_{n-1}\to...\to
X_1\xto{\pa_1}X_0\xto{\pa_0}A)$ the nonnegative chain complex
$sh(X)$ defined by $sh(X)_0=A$, $sh(X)_n=X_{n-1}$,
$sh(\pa)_n=\pa_{n-1}$, $n\ge 1$), we get functors
$$
\wh I=sh\circ \,aI:apres(\sA)\to Ch_{\ge 0}(\sA)
$$
and
$$
\wh J=sh\circ \,aJ:apres(\sA)\to Ch_{\ge 0}(\sA)
$$
the restrictions of which to the category of augmented pseudosimplicial  objects are used in the next section.

Next recall [12] the definition of the homotopy groups of a Kan cubical set.

A cubical set $X$ is said to be Kan if for any $n\ge1$, any
$(i,\al)\in\{1,\dots,n\}\times\{0,1\}$ and any collection of $2n-1$
elements $x_j^\ve\in X_{n-1}$, $1\le j\le n$, $\ve\in\{0,1\}$,
$(j,\ve)\ne(i,\al)$, satisfying
$$
\pa_j^\ve x_k^\om=\pa_{k-1}^\om x_j^\ve,\;\;1\le j<k\le
n,\;\;\ve,\om\in\{0,1\},\ \ (j,\ve)\neq(i,\al),\;\;(k,\om)\neq(i,\al)
$$
there exists $x\in X_n$ such that
$$
\pa_j^\ve x=x_j^\ve,\quad 1\le j\le n,\quad
\ve\in\{0,1\},\quad(j,\ve)\ne(i,\al).
$$

Let $X$ be a cubical set, $z\in X_n$ and $x_j^\ve\in X_{n-1}$, $n\ge 1$, $1\le j\le n$, $\ve\in\{0,1\}$. We write
$$\pa z=\begin{pmatrix}x_1^0&x_2^0&\cdots&x_n^0\\
x_1^1&x_2^1&\cdots&x_n^1
\end{pmatrix}$$
iff $\pa_j^\ve z=x_j^\ve$, $1\le j\le n$, $\ve\in\{0,1\}$.

Suppose now that $(X,\psi)$ is a Kan cubical set with a basepoint
$\psi$. Let
$$\wt{X}_0=X_0\;\;\text{and}\;\;\wt{X}_n=\left\{x\in X_n\,\Big|\,\pa x=\begin{pmatrix}\psi&\cdots&\psi\\
\psi&\cdots&\psi\end{pmatrix}\right\},\;\;n>0,$$
and define an equivalence relation $\sim$ on $\wt{X}_n$ $(n\ge 0)$ by
$$x\sim y\;\;\text{iff}\;\;\exists\;z\in X_{n+1}\;\;\text{such that}\;\;
\pa z=\begin{pmatrix}x&\psi&\cdots&\psi\\
y&\psi&\cdots&\psi\end{pmatrix},$$
or equivalently, by
$$x\sim y\;\;\text{iff}\;\;\exists\;w\in X_{n+1}\;\;\text{such that}\;\;
\pa w=\begin{pmatrix}\psi&\cdots&\psi&x\\
\psi&\cdots&\psi&y\end{pmatrix}$$
[12, Theorem 6]. Denote $\pi_n(X,\psi)=\wt{X}_n/\sim$, $n\ge 0$.

Let $[x],[y]\in\pi_n(X,\psi)$ $([x]=cl_\sim(x))$, $n\ge 1$. Since
$X$ is Kan, there is  $z\in X_{n+1}$ with
$$\pa z=\begin{pmatrix}x&\psi&\psi&\cdots&\psi\\
\pa_1^1z&y&\psi&\cdots&\psi\end{pmatrix}.$$ Clearly,
$\pa_1^1z\in\wt{X}_n$. Define $[x]\bl[y]=[\pa_1^1z]$. This
definition depends only on the equivalence classes of $x$ and $y$
and, with this multiplication, $\pi_n(X,\psi)$, $n\ge 1$, is a
group, called the $n$-th homotopy group of $(X,\psi)$.

Finally let us note that in the text we will freely make use of the Freyd-Mitchell embedding theorem (see e.~g. [20, p.25]) when applying various results in the literature about modules to objects in general abelian categories.

\section{Moore chain complex}\label{Moore}

\begin{prop}
Let $G$ be a cubical group and define
$$M_0(G)=G_0,\;\;\;M_n(G)=\left(\bigcap_{i=1}^n\Ker\pa_i^1\right)\cap
\left(\bigcap_{i=1}^{n-1}\Ker\pa_i^0\right),\;\;\;n>0.$$
Then:
\begin{enumerate}
\item[\rm{(a)}] $\pa_{n+1}^0(M_{n+1}(G))\sbs M_n(G)$, $n\ge 0$.
\item[\rm{(b)}]
$\Img\left(M_{n+1}(G)\xto{\pa_{n+1}^0}M_n(G)\right)\subseteq
\Ker\left(M_n(G)\xto{\pa_n^0}M_{n-1}(G)\right)$, $n>0$.
\item[\rm{(c)}] $\pa_{n+1}^0(M_{n+1}(G))$ is a normal subgroup of $M_n(G)$ and of $G_n$, $n\ge 0$.
\end{enumerate}
\end{prop}

\begin{proof}
Let $x\in M_{n+1}(G)$. By definition
\begin{align*}
& \pa_i^1\,\pa_{n+1}^0x=\pa_n^0\,\pa_i^1x=1,\quad 1\le i \le n,\\
& \pa_i^0\,\pa_{n+1}^0x=\pa_n^0\,\pa_i^0x=1,\quad 1\le i \le n-1.
\end{align*}
Consequently, $\pa_{n+1}^0x\in M_n(G)$. Hence (a) is proved.

Let $x\in M_{n+1}(G)$. Then $\pa_n^0x=1$. Using this we get $\pa_n^0\,\pa_{n+1}^0x=\pa_n^0\,\pa_n^0x=1$. Hence (b) is proved.

Suppose $y\in M_{n+1}(G)$ and $z\in G_n$. Obviously one has $s_{n+1}z\cdot y\cdot s_{n+1}z^{-1}\in M_{n+1}(G)$ and $\pa_{n+1}^0(s_{n+1}z\cdot y\cdot s_{n+1}z^{-1})=z\cdot \pa_{n+1}^0 y\cdot z^{-1}$. This proves (c).
\end{proof}

Thus, for any cubical group $G$,
$$
M(G)=(...\to M_n(G)\xto{\pa_n^0}M_{n-1}(G)\to...\to
M_1(G)\xto{\pa_1^0}M_0(G)\to1)
$$
is a chain complex of (not necessarily abelian) groups. We call
$M(G)$ the Moore chain complex of $G$. It is obvious that $M(G)$ and
its homology groups $H_n(M(G))=\Ker\,\pa_n^0\,/\Img\,\pa_{n+1}^0$
functorially depend on $G$.

Obviously, in the construction of $M(G)$ one may replace $G$ by a
precubical object $X$ in an abelian category (in [16] $M(G)$ is
introduced for cubical objects in an abelian category). In this case
we redefine the differential $\pa$ on $M(X)$ by
$\pa=(-1)^{n+1}\pa_n^0$. Then $M(X)$ is a chain subcomplex of $N(X)$
and one has a natural monomorphism $i:M(X)\to N(X)$. \vskip+4mm

Now let us introduce pseudocubical objects and pseudocubical
objects with pseudoconnections.

\vskip+4mm

\begin{defi}
A pseudocubical object $X$ in a category $\sA$ is a family
of objects $(X_n\in\sA)_{n\ge 0}$ together with  face
$\sA$-morphisms
\begin{align*}
\pa_i^0,\pa_i^1:&\ X_n\to X_{n-1}\ \ (1\le i\le n)
\intertext{and pseudodegeneracy $\sA$-morphisms}%
s_i:&\ X_{n-1}\to X_n\ \ (1\le i\le n)
\intertext{satisfying}%
\pa_i^\al\pa_j^\ve&=\pa_{j-1}^\ve\pa_i^\al\qquad
i<j,\quad\al,\ve\in\{0,1\}
\intertext{and}%
\pa_i^\al s_j&=\begin{cases}
s_{j-1}\pa_i^\al\ &i<j,\\
\id\ &i=j,\\
s_j\pa_{i-1}^\al\ &i>j,
\end{cases}
\end{align*}
for $\al\in\{0,1\}$.
\end{defi}

\begin{defi}
We say that a pseudocubical object $X$ in a category $\sA$
has  pseudoconnections if there are given  $\sA$-morphisms
$$
\Gm_i:X_n\to X_{n+1}\ \ (1\le i\le n)
$$
which satisfy
$$
\pa_i^\al\Gm_j=\begin{cases}
\Gm_{j-1}\pa_i^\al\quad&i<j,\;\;\al\in\{0,1\},\\
\id\quad&i=j,\;j+1,\;\;\al=0,\\
s_j\pa_j^\al\quad&i=j,\;j+1,\;\;\al=1,\\
\Gm_j\pa_{i-1}^\al\quad&i>j+1,\;\;\al\in\{0,1\}.
\end{cases}
$$
\end{defi}

Morphisms of pseudocubical objects and pseudocubical objects with pseudoconnections are defined in an obvious way, just as for cubical objects and cubical objects with connections.

For a category $\sA$, let us denote by $pcpc(\sA)$ the
category of pseudocubical objects with pseudoconnections in
$\sA$.

Let $\sA$ be an abelian category. Define the functor
$$
\ol{N}:pcpc(\sA)\to aps(\sA)
$$
as follows:
\begin{align*}
\ol{N}(X)_{-1}=X_0,\;\;&\ol{N}(X)_n=\bigcap_{i=1}^{n+1}
\Ker(\pa_i^1:X_{n+1}\to X_n),\;\;n\ge0,\\
\ve=\restr{\pa_1^0}{\ol{N}(X)_0}:&\ \ol{N}(X)_0\to\ol{N}(X)_{-1},\\
\pa_i=\restr{\pa_{i+1}^0}{\ol{N}(X)_n}:&\ \ol{N}(X)_n\to\ol{N}(X)_{n-1},\;\;n\ge 1,\;\;0\le i\le n,\\
s_j=\restr{\Gm_{j+1}}{\ol{N}(X)_{n-1}}:&\ \ol{N}(X)_{n-1}\to\ol{N}(X)_n,\;\;n\ge 1,\;\;0\le j\le n-1,\\
\ol{N}(f:X\to Y)_n&=\restr{f_{n+1}}{\ol{N}(X)_n}.
\end{align*}
On the other hand, we have the functors $\wh I$ and $\wh J$ from
$aps(\sA)$ to $Ch_{\ge 0}(\sA)$. One can easily see that
$$
\wh I\circ\ol{N}=M\textrm{ and }\wh{J}\circ\ol{N}=N.
$$

Let $G$ be a simplicial object in an abelian category $\sA$. Then
the normalized chain complex $I(G)$ of $G$ is a chain subcomplex
of the unnormalized chain complex $J(G)$ of $G$. The Moore theorem
says that this inclusion is a chain homotopy equivalence. The
proof of this theorem [15, p.94] does not use the simplicial
identity: $s_is_j=s_{j+1}s_i$, $i\le j$. Therefore one may assert
that for any augmented  pseudosimplicial object $S$ in an abelian
category $\sA$, the natural monomorphism $\wh I(S)\to\wh J(S)$ is
a chain homotopy equivalence. Replacing now $S$ by $\ol{N}(X)$,
where $X$ is a pseudocubical object with pseudoconnections in
$\sA$, and using $\wh I\circ\ol{N}=M$ and $\wh J\circ\ol{N}=N$, we
get a cubical analog of the Moore theorem:

\begin{theo}\label{normtheo}
Let $X$ be a pseudocubical object with pseudoconnections in an abelian category. Then  the natural monomorphism $i:M(X)\to N(X)$ is  a chain homotopy equivalence and, therefore, $i_*:H_n(M(X))\to H_n(N(X))$ is an isomorphism for all $n\ge0$.\hfill\qed
\end{theo}

In fact $i:M(X)\to N(X)$ has a natural homotopy inverse $r:N(X)\to
M(X)$ with $ri=\id_{M(X)}$ (see the proof of Theorem 22.1 of
[15]).

\;

\;

Note that the natural monomorphism $ N(X) \hookrightarrow C(X)$
need not be a chain homotopy equivalence for a cubical object $X$
with connections in an abelian category.

\;

\;

As is known the prime examples of pseudosimplicial objects are the
projective resolutions used by Tierney and Vogel to define their
derived functors. In the next section we construct cubical analogs
of Tierney-Vogel's projective resolutions (in order to define our
cubical derived functors) and show that they have
pseudodegeneracies and pseudoconnections.

\

The material in the rest of this section is not needed in what
follows; however we believe it is interesting for its own sake.

\

Suppose that $X$ is a cubical object with connections in an
abelian category $\sA$. Then $\ol{N}(X)$ is an augmented
simplicial object. Denote by $F(X)$ the chain subcomplex of $N(X)$
generated by the images of the degeneracies of $\ol{N}(X)$. Since
the degeneracy morphisms of $\ol{N}(X)$ are defined by
$s_j=\restr{\Gm_{j+1}}{\ol{N}(X)_{n-1}}$, $n\ge 1$, $0\le j\le n-1$, one
has
$$F(X)_n=\Gm_1(N(X)_{n-1})+\dots+\Gm_{n-1}(N(X)_{n-1})\quad\text{for all}\quad n\ge 2.$$
Besides, $F(X)_0=0$ and $F(X)_1=0$.

\begin{theo}
Let $X$ be a cubical object with connections in an abelian category
$\sA$. Then:
\begin{enumerate}
\item[\rm{(a)}] $N(X)=M(X)\oplus F(X)$ and hence $M(X)$ is isomorphic to $N(X)/F(X)$.
\item[\rm(b)] The canonical projection $\pi:N(X)\to N(X)/F(X)$ is a chain homotopy equivalence.
\end{enumerate}
\end{theo}

\begin{proof}
As $\ol{N}(X)$ is an augmented simplicial object, one has
$$
\wh J(\ol{N}(X))=\wh{I}(\ol{N}(X))\oplus D(\ol{N}(X))
$$
(where $D(\ol{N}(X))$ is the chain subcomplex of $\wh
J(\ol{N}(X))$ generated by the degenerate elements of
$\ol{N}(X)$), and the canonical projection
$$
\wh J(\ol{N}(X))\to\wh J(\ol{N}(X))/D(\ol{N}(X))
$$
is a chain homotopy equivalence (see Corollary 22.2 and 22.3 of
[15]). But $\wh J(\ol{N}(X))=N(X)$, $\wh{I}(\ol{N}(X))=M(X)$ and
$D(\ol{N}(X))=F(X)$. Hence (a) and (b) hold.
\end{proof}

Suppose $X$ is a cubical set with connections. An n-cube $x\;
\epsilon\; X_n$ is said to be folded if there exists $y\;
\epsilon\; X_{n-1}$ such that $\Gamma_i y=x$ for some $i\;
\epsilon\; \{1,...,n-1\}$ (see [1]). When we define the singular
cubical homology $H_{\ast}(T)$ of a topological space $T$, it is
necessary to factor out the degenerate singular cubes. It easily
follows from Theorem 2.5 that the folded singular cubes can be
ignored in addition when we calculate $H_{\ast}(T)$.

Let $\sA$ be an abelian category. Restricting $M$ and $\ol{N}$ to
the category $cc(\sA)$, and $\wh{I}$ to the category $as(\sA)$, we
get the following commutative diagram
$$
\xymatrix{cc(\sA)\ar[rr]^-{M}\ar[rd]_-{\ol{N}}&&Ch_{\ge 0}(\sA)\\
&as(\sA)\ar[ru]_-{\wh{I}}}
$$
in which, by the Dold-Kan Theorem, $\wh{I}$ is an equivalence of
categories. Moreover, the functor $\ol{N}:cc(\sA)\to as(\sA)$ is
also an equivalence. The proof is very similar to the proof of the
Dold-Kan Theorem [20, p.270]. Thus we have got an alternative
proof of the result by Brown and Higgins [5] about the equivalence
of the categories $cc(\sA)$ and $Ch_{\ge 0}(\sA)$. Furthermore, we
have pointed out the functor $M:cc(\sA)\to Ch_{\ge 0}(\sA)$ which
realizes this equivalence.

\

In what follows, for pointed Kan cubical monoids $G$, we will always take the unit 1 of $G_0$ as basepoint, and denote $\pi_n(G,1)$ by $\pi_n(G)$.

\begin{prop}\label{moorekan}
Let $G$ be a Kan cubical monoid and suppose that $n\ge1$. Then:

\rm{(a)} For any $x,y\in\wt{G}_n$, \ $[x]\bullet[y]=[xy]$.

\rm{(b)} $\pi_n(G)$ is abelian.
\end{prop}

To prove this proposition, we need the following well known
\begin{lem}\label{twoopers}
Let $\bl$ and $*$ be binary operations on a set $E$. Assume that they have units and satisfy
$$
(a\bl b)*(a'\bl b')=(a*a')\bl(b*b')
$$
for all $a,a',b,b'\in E$. Then these operations coincide and $(E,\bl)$ is an abelian monoid.\hfill\qed
\end{lem}

\begin{proof}[Proof of Proposition \ref{moorekan}]
(cf. the proofs of Propositions 17.2 and 17.3 of [15]). Using the monoid structure on $G$, a binary operation may be defined on $\pi_n(G)$ by
$[x]*[y]=[xy]$. This is well-defined. Indeed, if $x,y\in\wt{G}_n$ then $xy\in\wt{G}_n$, and if $x\sim x'$ and $y\sim y'$, $x,x',y,y'\in\wt{G}_n$, i.~e.,
$$\pa z_1=\begin{pmatrix}x&1&\dots&1\\
x'&1&\dots&1\end{pmatrix}\;\;\text{and}\;\;\pa z_2=\begin{pmatrix}y&1&\dots&1\\
y'&1&\dots&1\end{pmatrix},\;\; z_1,z_2\in G_{n+1},$$
then
$$\pa(z_1z_2)=\begin{pmatrix}xy&1&\dots&1\\
x'y'&1&\dots&1\end{pmatrix},$$
i.e., $xy\sim x'y'$. Next, assume that $[x],[y],[u],[v]\in\pi_n(G)$. By the Kan condition, there are $w_1,w_2\in G_{n+1}$ with
$$\pa w_1=\begin{pmatrix}x&1&1&\dots&1\\
\pa_1^1w_1&y&1&\dots&1\end{pmatrix}\;\;\text{and}\;\;\pa w_2=\begin{pmatrix}u&1&1&\dots&1\\
\pa_1^1w_2&v&1&\dots&1\end{pmatrix},$$
whence
$$\pa(w_1w_2)=\begin{pmatrix}xu&1&1&\dots&1\\
\pa_1^1w_1\cdot\pa_1^1 w_2&yv&1&\dots&1\end{pmatrix}.$$
Then, by the definitions of $\bullet$ and $*$, we can write
\begin{gather*}
\left([x]\bl[y]\right)*\left([u]\bl[v]\right)=[\pa_1^1w_1]*[\pa_1^1w_2][\pa_1^1w_1\cdot\pa_1^1w_2]=\\
=[xu]\bl[yv]=\left([x]*[u]\right)\bl\left([y]*[v]\right).
\end{gather*}
Besides, $[x]\bl[1]=[x]=[1]\bl[x]$ and $[x]*[1]=[x]=[1]*[x]$. Consequently, we conclude, by Lemma \ref{twoopers}, that  $\bl$ and $*$ coincide and $\pi_n(G)$ is abelian.
\end{proof}

\begin{prop}\label{kanpimoore}
For any \emph{Kan} cubical group $G$,
$$
\pi_n(G) = H_n(M(G)),\ n\ge0.
$$
\end{prop}

\begin{proof}
Let us use the following notations:
\begin{align*}
&Z_0M(G)=M_0(G),\;Z_nM(G)=\Ker\left(M_n(G)\xto{\pa_n^0}M_{n-1}(G)\right),\ n>0,\\
&B_nM(G)=\Img\left(M_{n+1}(G)\xto{\pa_{n+1}^0}M_n(G)\right),\ n\ge0.
\end{align*}
It is evident that $Z_nM(G)=\wt{G}_n$ for all $n\ge 0$. Therefore,
in view of Proposition \ref{moorekan}, it suffices to show that $x\sim y$ if
and only if $xy^{-1}\in B_nM(G)$, $x,y\in\wt{G}_n=Z_nM(G)$. Suppose
that $x,y\in\wt{G}_n$ and $x\sim y$. Then
\begin{align*}
\pa z&=\begin{pmatrix} 1&\dots&1&x\\
1&\dots&1&y\end{pmatrix}
\intertext{for some $z\in G_{n+1}$. Let $u=z\cdot s_{n+1}\,y^{-1}$. %
One can easily check that}%
\pa u&=\begin{pmatrix} 1&\dots&1&xy^{-1}\\
1&\dots&1&1\end{pmatrix}.
\intertext{Hence $u\in M_{n+1}(G)$ and $\pa_{n+1}^0u=xy^{-1}$. That %
is $xy^{-1}\in B_nM(G)$. Conversely, assume that $x,y\in Z_nM(G)$ %
and $xy^{-1}\in B_nM(G)$. Then $\pa_{n+1}^0w=xy^{-1}$ for some $w\in %
M_{n+1}(G)$. Clearly,}%
\pa w&=\begin{pmatrix} 1&\dots&1&xy^{-1}\\
1&\dots&1&1\end{pmatrix}.
\end{align*}
This gives $[x]*[y^{-1}]=[xy^{-1}]=[1]$, whence $[x]=[y]$, i.~e., $x\sim y$.
\end{proof}

Any cubical group with connections is Kan [19]. Combining this with Theorem \ref{normtheo} and Proposition \ref{kanpimoore}, we get
\begin{coro}
For any cubical abelian group $G$ with connections, $\pi_n(G)$ is naturally isomorphic to $H_n(N(G))$ for all $n\ge0$.\hfill\qed
\end{coro}

\section{Cubical derived functors}

\begin{defi}
Let $\sA$ be a category and
$$
f_1^0,\dots,f_n^0,f_1^1,\dots,f_n^1:A\to B
$$
a sequence of $\sA$-morphisms, $n\ge 1$. A cubical kernel
of the sequence $( f_1^0,\dots,f_n^0,f_1^1,\dots,f_n^1)$ is a
sequence
$$
k_1^0,\dots,k_n^0,k_{n+1}^0,k_1^1,\dots,k_n^1,k_{n+1}^1:K\to A
$$
of $\sA$-morphisms such that
\begin{itemize}
\item[(i)]
$f_i^\om k_j^\al=f_{j-1}^\al k_i^\om$ for $1\le i<j\le n+1$,
$\om,\al\in\{0,1\}$;
\item[]
\item[(ii)]
if $h_1^0,\dots,h_n^0,h_{n+1}^0,h_1^1,\dots,h_n^1,h_{n+1}^1:D\to A$ is any other sequence satisfying identities $f_i^\om h_j^\al=f_{j-1}^\al
h_i^\om$ for $1\le i<j\le n+1$, $\om,\al\in\{0,1\}$, then there
exists a unique $\sA$-morphism $h:D\to K$ with $k_i^\om
h=h_i^\om$,  $1\le i\le n+1$, $\om\in\{0,1\}$.
\end{itemize}
\end{defi}

It  immediately follows from the definition that cubical kernels are unique up to isomorphism if they exist.

Suppose $\sA$ has finite limits and let
$\xymatrix{A\ar@<0.5ex>[r]^-{f_1^0}\ar@<-0.5ex>[r]_-{f_1^1}&B}$ be a
pair of $\sA$-morphisms. Consider the diagram
$$\xymatrix{B&A\ar[l]_-{f_1^1}\ar[r]^-{f_1^0}&B\\
A\ar[u]^-{f_1^0}\ar[d]_-{f_1^1}&&A\ar[u]_-{f_1^0}\ar[d]^-{f_1^1}\\
B&A\ar[l]_-{f_1^1}\ar[r]^-{f_1^0}&B.}$$
By assumption, we have the limit diagram
$$\xymatrix{B&A\ar[l]_-{f_1^1}\ar[r]^-{f_1^0}&B\\
A\ar[u]^-{f_1^0}\ar[d]_-{f_1^1}&K\ar[l]_-{k_2^1}\ar[u]^-{k_1^0}
\ar[r]^-{k_2^0}\ar[d]_-{k_1^1}&A\ar[u]_-{f_1^0}\ar[d]^-{f_1^1}\\
B&A\ar[l]_-{f_1^1}\ar[r]^-{f_1^0}&B,}$$ i.e. the sequence
$k_1^0,k_2^0,k_1^1,k_2^1:K\to A$ is a cubical kernel of the pair
$$
\xymatrix{A\ar@<0.5ex>[r]^-{f_1^0}\ar@<-0.5ex>[r]_-{f_1^1}&B}.
$$
Clearly, in fact, one has

\begin{prop}
If $\sA$ admits finite limits, then cubical kernels exist in $\sA$
for any sequence $(f_1^0,\dots,f_n^0,f_1^1,\dots,f_n^1)$ and any
$n\ge 1$.\hfill\qed
\end{prop}

Let $\cP$ be a class of objects of $\sA$. Recall [7] that
an $\sA$-morphism $f:A\to A'$ is said to be
$\cP$-epimorphic iff
$$\Hom_{\sA}(Q,f): \Hom_{\sA}(Q,A)\to\Hom_{\sA}(Q,A')$$
is surjective for all $Q\in\cP$. Also recall that $\cP$ is called a
projective class if for each $A\in\sA$ there exists a
$\cP$-epimorphism $e:Q\to A$ with $Q\in\cP$.

Let $\cP$ be a projective class in $\sA$,
$\xymatrix{X\ar[r]^-\pa&A}$ an augmented precubical object over
$A\in\sA$, and suppose that $\sA$ has finite
limits. By Proposition 3.2, we have a factorization
$$
X:\
\xymatrix@!C=1.8em@R=3em@M=1.2ex@L=.2ex{ \cdots\
X_{n+1}\ar@<1.5ex>[rr]^-{\pa_1^0}\ar@<-.9ex>[rr]_-{\pa_{n+1}^1}
\ar[dr]_-{e_{n+1}} &{{}^{{}^{\vdots}}}&\ X_n\
\ar@<1.5ex>[rr]^-{\pa_1^0}\ar@<-.9ex>[rr]_-{\pa_n^1}\ar[dr]_-{e_n}
&{{}^{{}^{\vdots}}}&**[r]X_{n-1}\ \cdots&\ \
X_1\ar@<.5ex>[rr]^-{\pa_1^0}\ar@<-.5ex>[rr]_-{\pa_1^1}\ar[dr]_-{e_1}
&&X_0\ar[r]^-{\pa}&A
\\
&K_{n+1}\ar@<1.2ex>[ur]^-{k_1^0}\ar@{}[ur]|-\ddots\ar@<-1.5ex>[ur]_-{k_{n+1}^1}
&&K_n\ar@<1.2ex>[ur]^-{k_1^0}\ar@{}[ur]|-\ddots\ar@<-1.5ex>[ur]_-{k_n^1}
&&&K_1\ar@<.5ex>[ur]^-{k_1^0}\ar@<-.5ex>[ur]_-{k_1^1}}
$$
where
$\xymatrix@1{K_1\ar@<0.5ex>[r]^-{k_1^0}\ar@<-0.5ex>[r]_-{k_1^1}&X_0}$
is a kernel pair of $\xymatrix{X_0\ar[r]^-\pa&A}$, and\vskip-1em
$$
k_1^0,\dots,k_n^0,k_1^1,\dots,k_n^1:K_n\to X_{n-1}
$$ a cubical kernel of $(\pa_1^0,\dots,\pa_{n-1}^0,\pa_1^1,\dots,\pa_{n-1}^1)$ for $n\ge
2$. We say that
\begin{itemize}
\item[1)] $\xymatrix{X\ar[r]^-\pa&A}$ is $\cP$-projective iff each
$X_n\in\cP$;
\item[2)] $\xymatrix{X\ar[r]^-\pa&A}$ is $\cP$-exact iff
$\pa$ and $e_n$ $(n\ge 1)$ are $\cP$-epimorphic;
\item[3)] $\xymatrix{X\ar[r]^-\pa&A}$ is $\cP$-projective resolution of $A$
iff it is $\cP$-projective and $\cP$-exact.
\end{itemize}

Obviously, if $\sA$ is a category with finite limits and a
projective class $\cP$, then each $A\in\sA$ has a
$\cP$-projective resolution. Moreover, the following comparison
theorem shows that such a resolution is unique up to precubical
homotopy equivalence.

\begin{theo}\label{precm}
Let $\xymatrix{X\ar[r]^-\pa&A}$ be $\cP$-projective and
$\xymatrix{X'\ar[r]^-{\pa'}&A'}$ be $\cP$-exact. Then any $\sA$-morphism $f:A\to A'$ can be extended to a precubical morphism
$$\xymatrix{X\ar[rr]^-\pa\ar[dd]_-{\ol{f}}&&A\ar[dd]^-f\\
&&&&\\
X'\ar[rr]^-{\pa'}&& A'}$$ over $f$ $($i.~e., $\ol{f}$ and $f$ form a
morphism of augmented precubical objects$)$. Furthermore, any two
such extensions are precubically homotopic. That is, if
$\ol{f},\ol{g}:X\to X'$ are two extensions of $f$, then there exist
$\sA$-morphisms $h_n:X_n\to X'_{n+1}$, $n\ge 0$, such that
\begin{align*}
&\pa_1^0h_n=f_n,\quad\pa_1^1h_n=g_n,\quad n\ge 0,\\
&\pa_i^\ve h_n=h_{n-1}\pa_{i-1}^\ve,\;\;n\ge 1,\;\;1<i\le n+1,\;\;\ve\in\{0,1\}.
\end{align*}
\end{theo}

\begin{proof}
We construct the extension $\ol{f}=(f_n:X_n\to X'_n)$ and show
its uniqness up to precubical homotopy by induction on $n$. Since
$X_0$ is $\cP$-projective and $\pa':X'_0\to A'$ is
$\cP$-epimorphic, there is $f_0:X_0\to X'_0$ with $\pa'f_0=f\pa$.
Next, one has $\pa' f_0\pa_1^0=f\pa \pa_1^0=f\pa\pa_1^1=\pa'
f_0\pa_1^1$. Therefore $ f_0\pa_1^0=k_1^0\vf_1$ and
$f_0\pa_1^1=k_1^1\vf_1$ for a uniquely defined $\vf_1:X_1\to K'_1$.
As $X_1$ is $\cP$-projective and $e_1:X'_1\to K'_1$ is
$\cP$-epimorphic, there exists  $f_1:X_1\to X'_1$ with
$e_1f_1=\vf_1$, and we have
$\pa_1^0f_1=k_1^0e_1f_1=k_1^0\vf_1=f_0\pa_1^0$ and
$\pa_1^1f_1=k_1^1e_1f_1=k_1^1\vf_1=f_0\pa_1^1$. Thus $f_0$ and $f_1$
are constructed. Inductively, suppose given $\sA$-morphisms
$f_r:X_r\to X'_r$ for $r\le n$ so that $\pa_i^\om
f_r=f_{r-1}\pa_i^\om$, $1\le i\le r$, $\om\in\{0,1\}$. Then
$\pa_i^\om f_n\pa_j^\al=f_{n-1}\pa_i^\om\pa_j^\al=f_{n-1}\pa_{j-1}^\al\pa_i^\om=\pa_{j-1}^\al
f_n\pa_i^\om$, $1\le i<j\le n+1$, $\om,\al\in\{0,1\}$. Hence there
is a unique $\vf_{n+1}:X_{n+1}\to K'_{n+1}$ such that
$k_i^\om\vf_{n+1}=f_n\pa_i^\om$, $1\le i\le n+1$, $\om\in\{0,1\}$.
Since $X_{n+1}$ is $\cP$-projective  and $e_{n+1}:X'_{n+1}\to
K'_{n+1}$ is $\cP$-epimorphic, there exists  $f_{n+1}:X_{n+1}\to
X'_{n+1}$ with $e_{n+1}f_{n+1}=\vf_{n+1}$. Then we have  $\pa_i^\om
f_{n+1}=k_i^\om e_{n+1}f_{n+1}=k_i^\om\vf_{n+1}=f_n\pa_i^\om$, $1\le
i\le n+1$, $\om\in\{0,1\}$. This completes the inductive step and
proves the existence of $\ol{f}$.

Now suppose $\ol{g}=(g_n:X_n\to X'_n)$ is another extension of $f:A\to A'$. We want to construct $h=(h_n:X_n\to X'_{n+1})$ with $\pa_1^0h_n=f_n$, $\pa_1^1h_n=g_n$, $n\ge 0$, and  $\pa_i^\ve h_n=h_{n-1}\pa_{i-1}^\ve$, $1<i\le n+1$, $\ve\in\{0,1\}$. For $n=0$ consider diagram
$$
\xymatrix@L=.2ex{&&**[r]{\ X_0}\ar@{-->}[ddll]_-{h_0}\ar@{-->}[dddl]_-{q_0}
\ar[rr]^-\pa\ar@<1.2ex>[dd]^-{g_0}\ar@<.2ex>[dd]_-{f_0}&&A\ar[dd]^-f\\
\\
X'_1\ar[dr]_-{e_1}&&**[r]{\ X'_0}\ar[rr]^-{\pa'}&&**[r]A'\;\;.\\
&**[r]K'_1\ar@<-1ex>[ur]^-{k_1^0}\ar@<-2ex>[ur]_-{k_1^1}&&
}
$$
As $\pa' f_0=\pa' g_0$, there exists $q_0:X_0\to K'_1$ such that
$k_1^0q_0=f_0$ and $k_1^1q_0=g_0$. Next, since $X_0$ is
$\cP$-projective and $e_1$ is $\cP$-epimorphic, there is $h_0:X_0\to
X'_1$ with $e_1h_0=q_0$. This and the two previous equalities give
$\pa_1^0h_0=k_1^0e_1h_0=k_1^0q_0=f_0$ and
$\pa_1^1h_0=k_1^1e_1h_0=k_1^1q_0=g_0$. Thus $h_0$ is constructed.
Inductively, suppose given $h_0,h_1,\dots,h_{n-1}$ with the required
properties.
$$
\xymatrix@L=.2ex@M=2ex{&&**[r]X_n\ar@{-->}[ddll]_-{h_n}\ar@{-->}[dddl]_-{q_n}
\ar@<1.5ex>[rr]^-{\pa_1^0}\ar@<-1ex>[rr]_-{\pa_n^1}
\ar@<1.4ex>[dd]^-{g_n}\ar@<-0.2ex>[dd]_-{f_n}
&{{}^{{}^{\vdots}}}&**[r]X_{n-1}
\ar@<-.2ex>[dd]_-{f_{n-1}}\ar@<1.4ex>[dd]^-{g_{n-1}}
\ar@<0.7ex>[ddll]_{h_{n-1}}\\
\\
**[r]X'_{n+1}\ar[dr]_-{e_{n+1}}&&**[r]X'_n
\ar@<1.5ex>[rr]^-{\pa_1^0}\ar@<-1ex>[rr]_-{\pa_n^1}
&{{}^{{}^{\vdots}}}&**[r]X'_{n-1}\ .\\
&**[r]K'_{n+1}
\ar@<1.1ex>[ur]^-{k_1^0}\ar@{}[ur]|-\ddots\ar@<-1.6ex>[ur]_-{k_{n+1}^1}
}
$$
Define $\psi_i^\ve:X_n\to X'_n$, $1\le i\le n+1$, $\ve\in\{0,1\}$,
as follows:
$$\psi_1^0=f_n,\;\;\psi_1^1=g_n,\;\;\psi_i^\ve=h_{n-1}\pa_{i-1}^\ve,\;\;1<i\le n+1,\;\;\ve\in\{0,1\}.$$
By the induction assumption,
\all
\begin{align*}
\pa_1^0\psi_j^\ve&=\pa_1^0h_{n-1}\pa_{j-1}^\ve=f_{n-1}\pa_{j-1}^\ve\pa_{j-1}^\ve f_n=\pa_{j-1}^\ve\psi_1^0,\\
&1<j\le n+1,\quad\ve\in\{0,1\},\quad n\ge 1,\\
\pa_1^1\psi_j^\ve&=\pa_1^1h_{n-1}\pa_{j-1}^\ve=g_{n-1}\pa_{j-1}^\ve\pa_{j-1}^\ve g_n=\pa_{j-1}^\ve\psi_1^1,\\
&1<j\le n+1,\quad\ve\in\{0,1\},\quad n\ge 1,\\
\pa_i^\al\psi_j^\ve&=\pa_i^\al h_{n-1}\pa_{j-1}^\ve=h_{n-2}\pa_{i-1}^\al\pa_{j-1}^\ve=h_{n-2}\pa_{j-2}^\ve \pa_{i-1}^\al=\pa_{j-1}^\ve h_{n-1}\pa_{i-1}^\al\\
&=\pa_{j-1}^\ve\psi_i^\al,\quad
1<i<j\le n+1,\quad\al,\ve\in\{0,1\},\quad n>1.
\end{align*}
That is $\pa_i^\al\psi_j^\ve=\pa_{j-1}^\ve\psi_i^\al$ for $1\le i<j\le n+1$, $\al,\ve\in\{0,1\}$. Therefore  $k_i^\ve q_n=\psi_i^\ve$,  $1\le i\le n+1$, $\ve\in\{0,1\}$, for a uniquely defined $q_n:X_n\to K'_{n+1}$. Since $X_n$ is $\cP$-projective and $e_{n+1}$ is $\cP$-epimorphic, there is $h_n:X_n\to X'_{n+1}$ with $e_{n+1}h_n=q_n$. Now we have
\begin{align*}
&\pa_1^0h_n=k_1^0e_{n+1}h_n=k_1^0q_n=\psi_1^0=f_n,\\
&\pa_1^1h_n=k_1^1e_{n+1}h_n=k_1^1q_n=\psi_1^1=g_n,\\
&\pa_i^\ve h_n=k_i^\ve e_{n+1}h_n=k_i^\ve q_n=\psi_i^\ve=h_{n-1}\pa_{i-1}^\ve,\\
&\quad\quad\quad 1<i\le n+1,\quad\ve\in\{0,1\}.
\end{align*}
This finishes the inductive step and completes the proof of the theorem.
\end{proof}

The following theorem is crucial for constructing our cubical
derived functors.

\begin{theo}\label{resdeg}
Suppose that
$$
P:\ \xymatrix@!C=1.8em@R=3em@M=1.2ex@L=.2ex{ \cdots\
P_{n+1}\ar@<1.5ex>[rr]^-{\pa_1^0}\ar@<-.9ex>[rr]_-{\pa_{n+1}^1}
\ar[dr]_-{e_{n+1}} &{{}^{{}^{\vdots}}}&\ P_n\
\ar@<1.5ex>[rr]^-{\pa_1^0}\ar@<-.9ex>[rr]_-{\pa_n^1}\ar[dr]_-{e_n}
&{{}^{{}^{\vdots}}}&**[r]P_{n-1}\ \cdots&\ \
P_1\ar@<.5ex>[rr]^-{\pa_1^0}\ar@<-.5ex>[rr]_-{\pa_1^1}\ar[dr]_-{e_1}
&&P_0\ar[r]^-{\pa}&A
\\
&K_{n+1}\ar@<1.2ex>[ur]^-{k_1^0}\ar@{}[ur]|-\ddots\ar@<-1.5ex>[ur]_-{k_{n+1}^1}
&&K_n\ar@<1.2ex>[ur]^-{k_1^0}\ar@{}[ur]|-\ddots\ar@<-1.5ex>[ur]_-{k_n^1}
&&&K_1\ar@<.5ex>[ur]^-{k_1^0}\ar@<-.5ex>[ur]_-{k_1^1}}
$$
is a $\cP$-projective resolution of $A\in\sA$. Then:

{\rm(a)} \ $P$ has pseudodegeneracy operators, i.e., there exist $s_i:P_n\to P_{n+1}$, \ $n\ge 0$, \ $1\le i\le n+1$, satisfying
\begin{align*}
\pa_i^\al s_j&=\begin{cases}s_{j-1}\pa_i^\al\quad&i<j,\\
\id\quad&i=j,\\
s_j\pa_{i-1}^\al\quad&i>j,
\end{cases}
\intertext{where $\al\in\{0,1\}$.}%
\intertext{{\rm(b)} For any
pseudodegeneracy operators $s_i:P_n\to P_{n+1}$, $n\ge 0$, $1\le
i\le n+1$, of $P$, there exist $\Gm_i:P_n\to P_{n+1}$, $n\ge 1$,
$1\le i\le n$, satisfying}
\pa_i^\al \Gm_j&=\begin{cases}\Gm_{j-1}\pa_i^\al\quad&i<j,\;\;\al\in\{0,1\},\\
\id\quad&i=j,\;j+1,\;\;\al=0,\\
s_j\pa_j^\al\quad&i=j,\;j+1,\;\;\al=1,\\
\Gm_j\pa_{i-1}^\al\quad&i>j+1,\;\;\al\in\{0,1\}.
\end{cases}
\end{align*}
\end{theo}

As an immediate consequence we have

\begin{coro}\label{psec}
Any $\cP$-projective resolution $\xymatrix{P\ar[r]^-{\pa}&A}$ is an augmented pseudocubical object with pseudoconnections.\hfill\qed
\end{coro}

\begin{proof}[Proof of Theorem \ref{resdeg}] (a) As
$\xymatrix{K_1\ar@<0.5ex>[r]^-{k_1^0}\ar@<-0.5ex>[r]_-{k_1^1}&P_0}$
is a kernel pair of $P_0\xto\pa A$, there is $\vf_1:P_0\to K_1$ such
that $k_1^0\vf_1=\id$ and $k_1^1\vf_1=\id$. On the other hand, since
$P_0$ is $\cP$-projective and $e_1$ is $\cP$-epimorphic,
$e_1s_1=\vf_1$ for some $s_1:P_0\to P_1$ and we have
$\pa_1^0s_1=k_1^0e_1s_1=k_1^0\vf_1=\id$ and
$\pa_1^1s_1=k_1^1e_1s_1=k_1^1\vf_1=\id$. Thus $s_1:P_0\to P_1$ is
constructed. Inductively, suppose given $s_1:P_0\to P_1$,
$s_1,s_2:P_1\to P_2,\dots,s_1,\dots,s_n:P_{n-1}\to P_n$ with the
required properties. Fix $j$, $1\le j\le n+1$, and define
$\psi_{ij}^\ve:P_n\to P_n$, $1\le i\le n+1$, $\ve\in\{0,1\}$, by
$$
\psi_{ij}^\ve=\begin{cases}
\id\quad &i=j,\\
s_{j-1}\pa_i^\ve\quad&i<j,\\
s_j\pa_{i-1}^\ve\quad& i>j.
\end{cases}
$$
Using the induction assumption, one checks that $\pa_i^\ve\psi_{mj}^\al=\pa_{m-1}^\al\psi_{ij}^\ve$ for $1\le i<m\le n+1$, $\ve,\al\in\{0,1\}$. Then there  exists $\vf_j:P_n\to K_{n+1}$ such that $k_i^\ve\vf_j=\psi_{ij}^\ve$, $1\le i\le n+1$, $\ve\in\{0,1\}$. Since $P_n$ is
$\cP$-projective and $e_{n+1}$ is $\cP$-epimorphic, there is $s_j:P_n\to P_{n+1}$ with $e_{n+1}s_j=\vf_j$. According to this, for $i=1,\dots,n+1$ and $\ve\in\{0,1\}$, one has
$$
\pa_i^\ve s_j=k_1^\ve e_{n+1}s_j=k_i^\ve\vf_j=\psi_{ij}^\ve=\begin{cases}
\id\quad &i=j,\\
s_{j-1}\pa_i^\ve\quad&i<j,\\
s_j\pa_{i-1}^\ve\quad& i>j.
\end{cases}
$$
Thus we have constructed $s_1,\dots,s_{n+1}:P_n\to P_{n+1}$ with the desired properties.

(b) Define $\lb_1^0,\lb_2^0,\lb_1^1,\lb_2^1:P_1\to  P_1$ by  $\lb_1^0=\id$, $\lb_2^0=\id$, $\lb_1^1=s_1\pa_1^1$ and $\lb_2^1=s_1\pa_1^1$. One verifies that $\pa_1^\ve\lb_2^\al=\pa_1^\al\lb_1^\ve$, $\ve\in\{0,1\}$. Hence there exists $\mu_1:P_1\to K_2$ such that $k_i^\ve\mu_1=\lb_i^\ve$, $i=1,2$, $\ve\in\{0,1\}$. As $P_1$ is $\cP$-projective and $e_2$ is $\cP$-epimorphic, $\mu_1=e_2\Gm_1$ for some $\Gm_1:P_1\to P_2$, and we have
$$
\pa_i^\ve \Gm_1=k_i^\ve e_2\Gm_1=k_i^\ve\mu_1=\lb_i^\ve=\begin{cases}
\id\quad &i=1,2,\;\;\ve=0,\\
s_1\pa_1^1\quad&i=1,2,\;\;\;\ve=1.
\end{cases}
$$
Next, assume that $\Gm_1:P_1\to P_2$, $\Gm_1,\Gm_2:P_2\to P_3,\dots,\Gm_1,\dots,\Gm_{n-1}:P_{n-1}\to P_n$ with the required properties are constructed, and for any fixed $j$, $1\le j\le n$, define $\lb_{ij}^\ve:P_n\to P_n$, $1\le i\le n+1$, $\ve\in\{0,1\}$, as follows:
$$
\lb_{ij}^\ve=\begin{cases}\Gm_{j-1}\pa_i^\ve\quad& i<j,\;\;\ve\in\{0,1\},\\
\id\quad &i=j,\;j+1,\;\;\ve=0,\\
s_j\pa_j^\ve\quad&i=j,\;j+1,\;\;\ve=1,\\
\Gm_j\pa_{i-1}^\ve\quad& i>j+1,\;\;\ve\in\{0,1\}.
\end{cases}
$$
By the induction assumption, $\pa_i^\ve\lb_{mj}^\al=\pa_{m-1}^\al\lb_{ij}^\ve$, $1\le i<m\le n+1$, $\ve,\al\in\{0,1\}$. Consequently, there is $\mu_j:P_n\to K_{n+1}$ such that, $k_i^\ve\mu_j=\lb_{ij}^\ve$, $1\le i\le n+1$, $\ve\in\{0,1\}$.

Since $P_n$ is $\cP$-projective and $e_{n+1}$ is $\cP$-epimorphic, there exists $\Gm_j:P_n\to P_{n+1}$ with $e_{n+1}\Gm_j=\mu_j$. Then, for $i=1,\dots,n+1$ and $\ve\in\{0,1\}$, one has
$$
\pa_i^\ve \Gm_j=k_i^\ve e_{n+1}\Gm_j=k_i^\ve\mu_j=\lb_{ij}^\ve\begin{cases}\Gm_{j-1}\pa_i^\ve\quad&i<j,\;\;\ve\in\{0,1\},\\
\id\quad&i=j,\;j+1,\;\;\ve=0,\\
s_j\pa_j^\ve\quad&i=j,\;j+1,\;\;\ve=1,\\
\Gm_j\pa_{i-1}^\ve\quad&i>j+1,\;\;\ve\in\{0,1\}.
\end{cases}
$$
Thus we have constructed $\Gm_1,\dots,\Gm_n:P_n\to P_{n+1}$
satisfying the desired properties.
\end{proof}

Let $\sA$ be an abelian category. It is shown in [16] that the
functor $N:c(\sA)\to Ch_{\ge 0}(\sA)$ sends
cubically homotopic morphisms to chain homotopic  morphisms. The
proof, which we repeat here because of the completeness, shows that
in fact one has

\begin{prop}\label{prechN}
Let $X$ be a pseudocubical object and $Y$ a precubical object in an abelian category
$\sA$, and $f,g:X\to Y$ precubical morphisms. If $f$ and
$g$ are precubically homotopic $(${\rm see Theorem} $3.3)$, then
$N(f)$ and $N(g):N(X)\to N(Y)$ are chain homotopic.
\end{prop}

\begin{proof}
Let $h=(h_n:X_n\to Y_{n+1})_{n\ge0}$ be a precubical homotopy from $f$ to $g$, and let $\nu_n$ denote the canonical monomorphism from $N_n(X)$ to $X_n$, $n\ge 0$. One checks that $\pa_i^1(h_n-g_{n+1}s_1)\nu_n=0$, $1\le j\le n+1$, $n\ge 0$. Consequently, we have morphisms
$$
t_n=(h_n-g_{n+1}s_1)\nu_n:N_n(X)\to N_{n+1}(Y),\quad n\ge 0.
$$
Clearly,
$$
\pa t_0=\pa_1^0(h_0-g_1s_1)\nu_0=(f_0-g_0)\nu_0=N_0(f)-N_0(g).
$$
Further, for all $n\ge 1$, we have
\all
\begin{align*}
\pa t_n&+t_{n-1}\pa=\left(\sum_{i=1}^{n+1}(-1)^{i+1}\pa_i^0\right)(h_n-g_{n+1}s_1)\nu_n+(h_{n-1}-g_ns_1)\left(\sum_{i=1}^n(-1)^{i+1}\pa_i^0\right)\nu_n\\
&=f_n\nu_n-g_n\nu_n+\left(\sum_{i=2}^{n+1}(-1)^{i+1}\pa_i^0\right)
(h_n-g_{n+1}s_1)\nu_n+\left(\sum_{i=1}^n(-1)^{i+1}\pa_{i+1}^0\right)(h_n-g_{n+1}s_1)\nu_n\\
&=f_n\nu_n-g_n\nu_n+\left(\sum_{i=2}^{n+1}(-1)^{i+1}\pa_i^0\right)(h_n-g_{n+1}s_1)
\nu_n+\left(\sum_{i=2}^{n+1}(-1)^i\pa_i^0\right)(h_n-g_{n+1}s_1)\nu_n\\
&=f_n\nu_n-g_n\nu_n=N_n(f)-N_n(g).
\end{align*}
Hence $t=(t_n:N_n(X)\to N_{n+1}(Y))_{n\ge0}$ is a chain homotopy from $N(f)$ to $N(g)$.
\end{proof}

This proposition together with Theorem 2.4 implies

\begin{prop}\label{prechM}
Let $X$ and $Y$ be pseudocubical objects with pseudoconnections in
an abelian category $\sA$, and $f,g:X\to Y$ precubical
morphisms. If $f$ and $g$ are precubically homotopic, then $M(f)$
and $M(g):M(X)\to M(Y)$ are chain homotopic.\hfill\qed
\end{prop}

We are now ready to introduce cubical derived functors.

Let $\sA$ be a category with finite limits and a projective
class $\cP$, ${\mathscr B}$ an abelian category, and $T:\sA\to{\mathscr B}$ an arbitrary (covariant) functor. We construct
the left cubical derived functors $L_n^\cub T:\sA\to{\mathscr B}$, $n\ge 0$, of $T$ as follows. If $A\in\sA$, choose (once and for all) a $\cP$-projective resolution $P\to
A$ and define
$$L_n^\cub T(A)=H_n\left(N(T(P))\right),\quad n\ge 0.$$
Theorem 3.3 and Propositions 3.4(a) and 3.6 show that the objects  $L_n^\cub T(A)$ are independent (up to natural isomorphism) of the
resolution chosen (if $P'\to A$ is a second $\cP$-projective
resolution and $f:P\to P'$ a precubical morphism extending $\id:A\to
A$, then $H_n(N(T(f))):H_n(N(T(P)))\to H_n(N(T(P')))$ are
isomorphisms). Moreover, by the same statements, it is immediately
clear that $L_n^\cub T$ are defined on morphisms and are functors
from $\sA$ to ${\mathscr B}$. Besides, functoriality in the
variable $T$ is obvious.

Similarly, in view of Theorem 3.3, Corollary 3.5 and Proposition
3.7, one can define functors $\wt{L}{\,}_n^\cub T:\sA\to{\mathscr B}$ by
$$\wt{L}{\,}_n^\cub T(A)=H_n\left(M(T(P))\right),\quad n\ge 0. $$
It follows from Theorem 2.4 that in fact there are isomorphisms
$$L_n^\cub T(A)\cong\wt{L}{\,}_n^\cub T(A),\quad n\ge 0, $$
which are natural in $A$ and in $T$.

\begin{remk}
As one sees the construction of $L_n^\cub T$ essentially uses \ref{resdeg}(a), and similarly $\wt L_n^\cub T$ essentially uses \ref{psec}. This contrasts with the fact that the construction of the derived functors by Tierney and Vogel does not use existence of pseudodegeneracies in $\cP$-projective presimplicial resolutions. On the other hand, as shown in [16], the functor $C$ (see Section 1) sends precubically homotopic morphisms of precubical objects to chain homotopic morphisms (cf. Propositions \ref{prechN} and \ref{prechM}). This together with Theorem \ref{precm} allows us to conclude that one does not need Proposition \ref{resdeg}(a) to prove that the functors
$$
{\bar L}_n^\cub T(A)=H_n(C(T(P))),\ n\ge0,
$$
where $P$ is a $\cP$-projective precubical resolution augmented over $A\in\sA$, are correctly defined. But in this way one obtains ``bad'' derived functors by the following reason. One can easily see that ${\bar L}_n^\cub T(Q)\cong T(Q)$ for any $Q\in\cP$ and any $n\ge0$. Thus in general the higher ($n>0$) derived functors ${\bar L}_n^\cub T$ do not vanish on $\cP$-projectives, i.~e., the crucial property of derived functors is not satisfied. In particular, there is no chance for ${\bar L}_n^\cub T$ to coincide with the classical derived functors for additive functors $T$ between abelian categories. At the same time, the derived functors $L_n^\cub T$ and $\wt L_n^\cub T$ certainly vanish on $\cP$-projectives for $n>0$, and in particular coincide with the classical derived functors for additive $T$, as we will show below in Theorem \ref{coinc}.
\end{remk}

\section{The case of an additive functor between abelian categories}

Suppose that $\sA$ is an abelian category and
$T:\sA\to{\mathscr B}$ an additive functor. Our aim is to compare $L_n^\cub
T$ $(n\ge 0)$ with the derived functors of $T$ in the sense
of Eilenberg-Moore [7].

\begin{prop}
Let $\sA$ be an abelian category with a projective class
$\cP$ and suppose that $\xymatrix{X\ar[r]^-{\pa}&A}$ is a precubical
object in $\sA$ augmented over $A\in\sA$. If
$\xymatrix{X\ar[r]^-{\pa}&A}$ is $\cP$-exact, then its augmented
Moore chain complex
$$
M(X):\cdots\to M_n(X)\xto{(-1)^{n+1}\pa_n^0}M_{n-1}(X)\to\dots\to
M_1(X)\xto{\pa_1^0}M_0(X)\xto\pa A\to0
$$
is $\cP$-exact in the sense of Eilenberg-Moore $[7]$. That is, the sequence of abelian groups
\begin{multline*}
\cdots\to\Hom_{\sA}(Q,M_n(X))\xto{d_n}\Hom_{\sA}(Q,M_{n-1}(X))\to\dots\\
\dots\to\Hom_{\sA}(Q,M_0(X))\xto{d_0}\Hom_{\sA}(Q,A)\to0,
\end{multline*}
where  $d_0=\Hom_{\sA}(Q,\pa)$ and
$d_n=\Hom_{\sA}(Q,(-1)^{n+1}\pa_n^0)$ $(n\ge 1)$, is exact for any
$Q\in\cP$.
\end{prop}

\begin{proof}
Evidently, $d_nd_{n+1}=0$, $n\ge 0$. Besides, since $\pa:X_0\to A$
is $\cP$-epimorphic, $d_0$ is surjective. Let $g:Q\to X_0=M_0(X)$
be an $\sA$-morphism with $d_0(g)=0$ and consider the
diagram
$$
\xymatrix@L=.2ex@M=2ex{
&&Q\ar@{-->}[ddll]_-{g'}\ar@{-->}@`{(-60,-20)}[dddl]_-\varphi\ar@<-1ex>[dd]_-{g}\ar@<1ex>[dd]^-0\\
\\
X_1\ar@<0.8ex>[rr]^-{\pa_1^0}\ar@<-0.9ex>[rr]_-{\pa_1^1}
\ar[dr]_-{e_1}&&X_0\ar[r]^-\pa&A,\\
&K_1\ar@<1ex>[ur]^-{k_1^0}\ar@<-1ex>[ur]_-{k_1^1} }
$$
where $\xymatrix{K_1\ar@<0.5ex>[r]^-{k_1^0}\ar@<-0.5ex>[r]_-{k_1^1}&X_0}$
is a kernel pair of $\pa:X_0\to A$, and $e_1$ is $\cP$-epimorphic.
As $\pa g=0$, there exists a unique $\vf:Q\to K_1$ such that
$k_1^0\vf=g$ and  $k_1^1\vf=0$. But $\vf=e_1g'$ for some $g':Q\to
X_1$, and we have $\pa_1^0 g'=k_1^0e_1g'=k_1^0\vf=g$ and  $\pa_1^1
g'=k_1^1e_1g'=k_1^1\vf=0$. It then follows from the construction of
$M(X)$ that $g'=j_1 g''$, where $j_1$ denotes the inclusion
$M_1(X)\hookrightarrow X_1$ and $g''$ is a  uniquely defined
$\sA$-morphism from $Q$ to $M_1(X)$. Clearly,
$g=\pa_1^0g'=\pa_1^0j_1g''=\pa_1^0g''=d_1(g'')$. Thus the sequence
is exact at $\Hom_{\sA}(Q,M_0(X))$.

Now assume that $n>0$ and $f:Q\to M_n(X)$ is an $\sA$-morphism with
$d_n(f)=0$. Denote the inclusion $M_n(X)\hookrightarrow X_n$ by
$j_n$ and consider the diagram
$$
\xymatrix@L=.2ex@M=2ex{
&&Q\ar@{-->}[ddll]_-{f'}\ar@{-->}@`{(-100,-20)}[dddl]_-\varphi\ar@<-1.4ex>[dd]_-{\varphi_1^0}
\ar@<1ex>[dd]^-{\varphi_{n+1}^1}&&\\
&&{{}^{{}^{\dots}}}\\
X_{n+1}\ar@<1.6ex>[rr]^-{\pa_1^0}\ar@<-.9ex>[rr]_-{\pa_{n+1}^1}
\ar[dr]_-{e_{n+1}}%
&{{}^{{}^{\vdots}}}
&X_n\ar@<1.6ex>[rr]^-{\pa_1^0}\ar@<-.9ex>[rr]_-{\pa_n^1}%
&{{}^{{}^{\vdots}}}
&X_{n-1}\\
&K_{n+1}\ar@<1.1ex>[ur]^-{k_1^0}\ar@<-1.9ex>[ur]_-{k_{n+1}^1}\ar@{}[ur]|-\ddots
}
$$
where $(k_1^0,\dots,k_{n+1}^0,k_1^1,\dots,k_{n+1}^1)$ is a cubical
kernel of $(\pa_1^0,\dots,\pa_n^0,\pa_1^1,\dots,\pa_n^1)$, $e_{n+1}$
is $\cP$-epi\-mor\-phic, and
$$\vf_i^\ve=\begin{cases}j_nf\quad&i=n+1,\;\;\ve=0,\\
0\quad&(i,\ve)\neq(n+1,0).\end{cases}$$
One checks that
$$\pa_i^\ve\vf_j^\al=\pa_{j-1}^\al\vf_i^\ve,\quad 1\le i<j\le n+1,\quad \ve,\al\in\{0,1\}.$$
Consequently, there exists a unique $\vf:Q\to K_{n+1}$ such that
$k_i^\ve\vf=\vf_i^\ve$, $1\le i\le n+1$,  $\ve\in\{0,1\}$. Since $e_{n+1}$ is $\cP$-epimorphic and $Q\in\cP$, one has $\vf=e_{n+1}f'$ for some $f':Q\to X_{n+1}$. But then
$$\pa_i^\ve f'=\begin{cases}j_nf\quad&(i,\ve)=(n+1,0),\\
0\quad&(i,\ve)\neq(n+1,0).\end{cases}$$ From this we conclude, by
the construction of $M(X)$, that $f'=j_{n+1}f''$ for a uniquely
defined $\sA$-morphism $f'':Q\to M_{n+1}(X)$. Clearly,
$j_nf=\pa_{n+1}^0f'=\pa_{n+1}^0j_{n+1}f''=j_n\pa_{n+1}^0f''$, whence
$f=\pa_{n+1}^0f''=(-1)^{n+2}\pa_{n+1}^0(-1)^{n+2}f''=d_{n+1}((-1)^{n+2}f'')$.
\end{proof}

Combining Proposition 4.1 with Theorem 2.4 we get

\begin{prop}
Let $\sA$ be an abelian category with a projective class
$\cP$ and suppose that $\xymatrix{X\ar[r]^-\pa&A}$ is a
pseudocubical object with pseudoconnections in $\sA$
augmented over  $A\in\sA$. If  $X\to A$ is $\cP$-exact,
then the augmented chain complex $\xymatrix{N(X)\ar[r]^-\pa&A}$ is
$\cP$-exact in the sense  of Eilenberg-Moore.\hfill\qed
\end{prop}

For any pseudocubical object $X$ in an abelian category $\sA$, one has $\sA$-morphisms $\sg_n^X:X_n\to X_n$ defined
by
$$\sg_0^X=\id,\;\sg_1^X=(\id-s_1\pa_1^1),\dots,\sg_n^X=(\id-s_1\pa_1^1)
\cdots(\id-s_n\pa_n^1),\dots\;.$$ It is immediate that
$\pa_j^1\sg_n^X=0$, $n>0$, $1\le j\le n$. Therefore, by the
construction of $N(X)$, each $\sg_n^X$ factors as
$$
\xymatrix@M=1ex{X_n\ar[rr]^-{\sg_n^X}\ar[dr]_-{\tau_n^X}&&X_n\\
&N_n(X)\ar@{^{(}->}[ur]_-{\nu_n^X}}
$$
where $\nu_n^X$ is the inclusion. One can easily see that $\tau_n^X\nu_n^X=\id$, i.e., $\tau_n^X$ is a retraction for each $n$ (see [16, 17]). This together with Corollary 3.5 and Proposition 4.2 gives

\begin{prop}\label{retr}
Let $\sA$ be an abelian category with a projective class
$\cP$ which is closed with respect to retracts. If
$\xymatrix{P\ar[r]^-\pa&A}$ is a $\cP$-projective precubical
resolution of $A\in\sA$, then the augmented chain complex
$\xymatrix{N(P)\ar[r]^-\pa&A}$ is a  $\cP$-projective resolution of
$A$ in the sense of Eilenberg-Moore $[7]$.\hfill\qed
\end{prop}

Now let $\sA$ be an abelian category with a projective
class $\cP$, ${\mathscr B}$ an abelian category, and $T:\sA\to{\mathscr B}$ an additive (covariant) functor. Then one
constructs, with respect to $\cP$, the classical left derived
functors $L_n T:\sA\to{\mathscr B}$, $n\ge 0$. On the
other hand, since any abelian category admits finite limits, we can
build $\cP$-projective precubical resolutions, and therefore can
construct the cubical left derived functors $L_n^\cub T:\sA\to{\mathscr B}$, $n\ge 0$.

\begin{theo}\label{coinc}
Suppose $\cP$ is closed $[7]$ or, more generally, is closed with respect to retracts. Then there are natural isomorphisms
$$
L_n^\cub T(A)\cong L_nT(A),\quad A\in\sA,\quad n\ge 0.
$$
\end{theo}

\begin{proof}
Let  $\xymatrix@1{P\ar[r]^\pa&A}$ be a $\cP$-projective precubical
resolution of $A\in\sA$. By Proposition 3.4(a), $P$ has
pseudodegeneracy operators. Then, as noted above, we have
$\sA$-morphisms $\sg_n^P:P_n\to P_n$, $\tau_n^P:P_n\to N_n(P)$ with
$\sg_n^P=\nu_n^P\tau_n^P$ and $\tau_n^P\nu_n^P=1$. The latter gives
$T(\tau_n^P)T(\nu_n^P)=T(\tau_n^P\nu_n^P)=T(\id)=\id$. Hence
$T(\tau_n^P):T(P_n)\to T(N_n(P))$ is an epimorphism and
$T(\nu_n^P):T(N_n(P))\to T(P_n)$ is a monomorphism. Since $T$ is
additive, $T(\sg_n^P)=\sg_n^{T(P)}$. Consider the commutative
diagram
$$
\xymatrix@M=1ex{&N_n(T(P))\;\;\;\;\ar@{^{(}->}[dr]^-{\nu_n^{T(P)}}\\
T(P_n)\ar[ur]^-{\tau_n^{T(P)}}\ar[dr]_-{T(\tau_n^P)}\ar[rr]^-{T(\sg_n^P)\sg_n^{T(P)}}&&T(P_n)\;\;\;\\
&T(N_n(P))\;\;.\ar[ur]_-{T(\nu_n^P)} }$$
As $\tau_n^{T(P)}$ and  $T(\tau_n^P)$ are epimorphisms, $\Imm(T(\nu_n^P))=\Imm(\nu_n^{T(P)})$. But $\nu_n^{T(P)}$ is an inclusion and $T(\nu_n^P)$ is a monomorphism. Hence $\Imm(\nu_n^{T(P)})=N_n(T(P))$ and
$$
T(\nu_n^P):T(N_n(P))\to N_n(T(P))
$$
is an isomorphism for each $n$. One can easily see that these isomorphisms commute with the differentials and induce natural isomorphisms on the homologies. In view of Proposition \ref{retr}, $\xymatrix{N(P)\ar[r]^-\pa&A}$ is a $\cP$-projective resolution of $A$ in the sense of Eilenberg-Moore. Consequently, $L_nT(A)=H_n(T(N(P)))$. On the other hand, by the definition of $L_n^\cub T$, $L_n^\cub T(A)=H_n(N(T(P)))$. Thus we have isomorphisms
$$
L_nT(A)\cong L_n^\cub T(A),\quad n\ge0,
$$
which are natural in $A$ and in $T$.
\end{proof}

There remains an open question here. We do not know whether the cubical derived functors introduced by us coincide with the Tierney-Vogel derived functors in full generality, or at least, with the Dold-Puppe derived functors in the particular case of an arbitrary (not necessarily additive) functor on an abelian category.

\

\begin{tabular}{l}
{\sc Mathematisches Institut}\\
{\sc Universit\"at Bonn}\\
{\sc Endenicher Allee 60}\\
{\sc 53115 Bonn, Germany}\\
\\[-10pt]
\emph{E-mail address}: irpatchk@math.uni-bonn.de
\end{tabular}


\begin{thebibliography}{20}

\bibitem{1}
Antolini, R.
Geometric realisations of cubical sets with connections, and classifying spaces of categories.
Appl. Categ. Struct. 10, No.5, 481-494 (2002).

\bibitem{2}
Antolini, R.; Wiest, B.
The singular cubical set of a topological space.
Math. Proc. Camb. Philos. Soc. 126, No.1, 149-154 (1999).

\bibitem{3}
Barr, M.; Beck, J.
Acyclic models and triples.
Proc. Conf. Categor. Algebra, La Jolla 1965, 336-343 (1966).

\bibitem{4}
Brown, R.; Higgins, P. J.
On the algebra of cubes.
J. Pure Appl. Algebra 21, 233-260 (1981).

\bibitem{5}
Brown, R.; Higgins, P. J.
Cubical Abelian groups with connections are equivalent to chain complexes.
Homology Homotopy Appl. 5, No.1, 49-52 (2003).

\bibitem{6}
Dold, A.; Puppe, D.
Homologie nicht-additiver Funktoren. Anwendungen.
Ann. Inst. Fourier 11, 201-312 (1961).

\bibitem{7}
Eilenberg, S.; Moore, J. C.
Foundations of relative homological algebra.
Mem. Am. Math. Soc. 55, 39 p. (1965).

\bibitem{8}
Golasinski, M.
On weak homotopy equivalence of simplicial and cubical nerve.
Demonstr. Math. 14, 889-899 (1981).

\bibitem{9}
Golasinski, M.
Homotopies of small categories.
Fundam. Math. 114, 209-217 (1981).

\bibitem{10}
Jardine, J. F.
Cubical homotopy theory: a beginning.
Preprint, avaliable online --- see the following link: \texttt{http://www.math.uwo.ca/\~{}jardine/papers/preprints/cubical2.pdf}

\bibitem{11}
Kadeishvili, T.; Saneblidze, S.
A cubical model for a fibration.
J. Pure Appl. Algebra 196, No. 2-3, 203-228 (2005).

\bibitem{12}
Kan, D. M.
Abstract homotopy. I.
Proc. Natl. Acad. Sci. USA 41, 1092-1096 (1955).

\bibitem{13}
Kan, D. M.
Abstract homotopy. II.
Proc. Natl. Acad. Sci. USA 42, 255-258 (1956).

\bibitem{14}
Massey, W. S.
Singular homology theory.
Graduate Texts in Mathematics, 70. New York Heidelberg Berlin: Springer- Verlag. XII, 265 p. (1980).

\bibitem{15}
May, J. P.
Simplicial objects in algebraic topology.
Princeton, N.J.-Toronto-London-Melbourne: D. van Nostrand Company, Inc. VI, 161 p. (1968).

\bibitem{16}
\'Swi\k atek, A.
Kategorie obiekt\'ow kostkowych, preprint nr 27, Toru\'n 1975.

\bibitem{17}
\'Swi\k atek, A.
Category of cubical objects and category of simplicial objects.
Comment. Math. Prace Mat. 22 (1980/81), no. 2, 307--316.

\bibitem{18}
Tierney, M.; Vogel, W.
Simplicial resolutions and derived functors.
Math. Z. 111, 1-14 (1968).

\bibitem{19}
Tonks, A. P.
Cubical groups which are Kan.
J. Pure Appl. Algebra 81, No.1, 83-87 (1992).

\bibitem{20}
Weibel, C. A.
An introduction to homological algebra.
Cambridge Studies in Advanced Mathematics. 38. Cambridge: Cambridge Univ. Press. xiv, 450 p. (1995).

\end{thebibliography}
\end{document}